\documentclass[12pt]{article}

\usepackage[dvips]{graphics}

\usepackage[dvips]{epsfig}
\usepackage{authblk}

\newtheorem{lemma}{Lemma}

\newtheorem{prop}[lemma]{Proposition}

\newtheorem{rem}{Remark}
\newtheorem{conj}[lemma]{Conjecture}
\newenvironment{proof}[1][Proof]{\textbf{#1.} }{\ \rule{0.5em}{0.5em}}

\usepackage{latexsym}
\usepackage{amsfonts}
\usepackage{amssymb}
\usepackage{amsmath}
\usepackage[english]{babel}

\def\G{\Gamma}

\begin{document}
\date{}
\title{Complexity computation for compact 3-manifolds \\ via crystallizations and Heegaard diagrams}

\bigskip

 \author[*]{Maria Rita CASALI}  \author[*]{Paola CRISTOFORI}
 \author[**]{Michele MULAZZANI}
\affil[*]{Dipartimento di Matematica, Universit\`a di Modena e
Reggio Emilia} \affil[**]{Dipartimento di Matematica, Universit\`a
di Bologna}
\renewcommand\Authands{ and }

\maketitle

\begin{abstract} The idea of computing Matveev complexity by using Heegaard decompositions has been
recently developed by two different approaches: the first one for
closed 3-manifolds via crystallization theory, yielding the notion
of {\it Gem-Matveev complexity}; the other one for compact
orientable 3-manifolds via generalized Heegaard diagrams, yielding
the notion of {\it modified Heegaard complexity}. In this paper we
extend to the non-orientable case the definition of modified
Heegaard complexity and prove that for closed 3-manifolds
Gem-Matveev complexity and modified Heegaard complexity coincide.
Hence, they turn out to be useful different tools to compute the
same upper bound for Matveev complexity.

\end{abstract}

\bigskip

\small{

\thanks{


\smallskip


}

\bigskip
\bigskip

\section{\hskip -0.7cm . Introduction}

In 1990 S. Matveev  proposed in \cite{[M$_1$]} to attack the
problem of studying systematically the whole set $\mathcal M$ of
compact 3-manifolds by choosing a suitable notion of {\it
complexity}, i.e. a non-negative function which filters $\mathcal
M$ and is able to {\it ``measure how complicated a combinatorial
description of the manifold must be"}. If the filtration has the
properties of finiteness (only a finite number of closed
irreducible 3-manifolds have a fixed complexity) and additivity
with respect to connected sum (the complexity of the connected sum
is the sum of the complexities of the summands), then it allows a
concrete catalogation of the elements of $\mathcal M$, via the
chosen combinatorial tool. In the same paper, Matveev introduced a
notion of complexity with the required properties, based on the
theory of simple spines (\cite{[M$_0$]} and \cite{[Pi]}).

We recall that a polyhedron $P$ embedded into a compact connected
3-manifold $M$ is called a \textit{spine} of  $M$ if $M$ (or $M$
minus an open 3-ball if $M$ is closed) collapses to $P$. Moreover,
a spine $S$ is said to be \textit{almost simple} if the link of
each point $x\in S$ can be embedded into  $K_4$, which is the
topological realization of the complete graph with four vertices.
A \textit{true vertex} of an almost simple spine $S$ is  a point
$x\in S$ whose link is homeomorphic to $K_4$.

The \textit{(Matveev) complexity} $c(M)$ of $M$ is defined as the
minimum number of true vertices among all almost simple spines of
$M$. The 3-sphere, the real projective space, the lens space
$L(3,1)$ and the spherical bundles $\mathbb S^1\times \mathbb S^2$
and $\mathbb S^1\tilde\times \mathbb S^2$ have complexity zero by
definition. Apart from these special cases, for a closed prime
manifold $M$, the complexity $c(M)$ turns out to be the minimum
number of tetrahedra needed to obtain $M$ via face paring of them
(\cite[Proposition 2]{[M$_1$]}, together with the related Remark).

During the last two decades, various authors produced tables of
closed 3-manifolds for increasing values of complexity, by simply
generating all triangulations (resp. spines) with a given number
of tetrahedra (resp. true vertices) and classifying topologically
the associated manifolds.  The obtained results concerning the
orientable (resp. non-orientable) case may be found in
\cite{[M$_1$]}, \cite{[O]}, \cite{[MP$_1$]}, \cite{[M$_3$]},
\cite{[M$_4$]}, \cite{[Ma]} and in the Web page
http://www.matlas.math.csu.ru/ (resp. in \cite{[B$_1$]},
\cite{[AM]}, \cite{[C$_4$]}, \cite{[AM_bis]} and
\cite[Appendix]{[B$_2$]}).

In general, the computation of the complexity of a given manifold
is a difficult problem (see \cite{[JRT1]} and \cite{[JRT2]} for
recent results). So, two-sided estimates of complexity become
important, especially when dealing with infinite families of
manifolds (see, for example, \cite{[M$_2$]}, \cite{[MPV]},
\cite{[PV]}). By \cite[Theorem 2.6.2]{[M$_2$]}, a lower bound  for
the complexity of a given manifold can be obtained from its first
homology group. Moreover, a lower bound for hyperbolic manifolds
can be obtained via volume arguments (see \cite{[M$_2$]},
\cite{[MPV]}, \cite{[PV]}). Upper bounds are easier to find, since
any pseudo-triangulation (or any spine) of $M$ obviously yields an
upper bound for $c(M)$.

The idea of computing Matveev complexity by using Heegaard
decompositions is already suggested in the foundational paper
\cite{[M$_1$]} by Matveev: from any Heegaard diagram $H = (S, v,w)$
of $M$, we can construct an almost simple spine of $M$ whose true
vertices are the intersection points of the curves of the two
systems $v$ and $w$, with the exception of those which lie on the
boundary of a region of $S-\{v \cup w\}$. In fact, the spine can
be obtained by adding to the surface $S$ the meridian disks
corresponding to the systems of curves and by removing the 2-disk
corresponding to an arbitrary region of $S-\{v \cup w\}$.

Starting from this idea, two different approaches to Matveev
complexity computation have been recently developed. The first
one, introduced in 2004 for closed 3-manifolds, is based on
crystallization theory; it has led to the notion of {\it
Gem-Matveev complexity}, $GM$-complexity for short (see
\cite{[C$_4$]}, together with subsequent papers \cite {[C$_5$]}
and \cite{[CC$_1$]}, or Section 3 of the present paper for a brief
account). Later, in 2010, the {\it modified Heegaard complexity}
($HM$-complexity) of a compact orientable 3-manifold has been
defined via generalized Heegaard diagrams (see \cite{[CMV]}). Both
invariants have been proved to be upper bounds for the Matveev
complexity.

>From the practical view-point, both $GM$-complexity and
$HM$-complexity have allowed to obtain estimations of complexity
for interesting classes of manifolds. In \cite{[C$_5$]}
$GM$-complexity has produced significant improvements in order to
estimate Matveev complexity for two-fold branched coverings of
$\mathbb S^3$, three-fold simple branched coverings of $\mathbb
S^3$ and 3-manifolds obtained by Dehn surgery on framed links in
$\mathbb S^3$. On the other hand, estimations for $n$-fold cyclic
coverings of $\mathbb S^3$ branched over 2-bridge knots and links,
torus knots and theta graphs, as well as for a wide class of
Seifert manifolds which generalize Neuwirth manifolds have been
obtained through $HM$-complexity in \cite{[CMV]}. Note also that,
in \cite{[C$_2$]}, $GM$-complexity has allowed us to complete the
classification of all non-orientable closed 3-manifold up to
complexity 6 (see \cite{[AM]} and \cite{[AM_bis]}).

The aim of the present paper is to extend the definition of modified
Heegaard complexity to the non-orientable case (Section 2), and to
prove that for each closed 3-manifold Gem-Matveev complexity and
modified Heegaard complexity coincide (Proposition 6).
Furthermore, experimental results concerning 3-manifolds admitting
a crystallization with ``few" vertices suggest 
equality between Matveev complexity and this upper bound, directly
computable via two apparently different methods for representing
3-manifolds (Conjecture 7).

\bigskip
\bigskip

\section{\hskip -0.7cm . Modified Heegaard complexity}

The notion of \textit{modified Heegaard complexity} for compact
orientable 3-manifolds (either with or without boundary) has been
introduced in \cite{[CMV]}, where a comparison with Matveev
complexity has been discussed. In this section we extend that
notion to the non-orientable case. In order to do that, some
preliminary definitions are required.

Let $\Sigma_g$ be either the closed, connected orientable surface of genus $g$ (with $g\ge 0$) or the closed,
connected non-orientable surface of genus $2g$ (with $g\ge 1$). So $\Sigma_g$ is the boundary of a
handlebody $\mathbb Y_g$ of genus $g$, $\mathbb Y_g$ being the orientable (resp. non orientable) 3-manifold
obtained from the $3$-ball $\mathbb D^3$ by adding $g$ orientable 1-handles (resp. $g$ 1-handles,
at least one of which is non-orientable).

A \textit{system of curves} on $\Sigma_g$ is a (possibly empty)
set of simple closed orientation-preserving\footnote{This means
that each curve $\gamma_i$ has an annular regular neighborhood, as
it always happens if $\Sigma_g$ is an orientable surface.} curves
$\mathcal C=\{\gamma_1,\ldots,\gamma_k\}$ on $\Sigma_g$ such that
$\gamma_i \cap \gamma_j = \emptyset$, for $1\le i\neq j\le k$.
Moreover, we denote with $V(\mathcal C)$ the set of connected
components of the surface obtained by cutting $\Sigma_g$ along the
curves of $\mathcal C$. The system $\mathcal C$ is said to be
\textit{proper} if all elements of $V(\mathcal C)$ have genus
zero, and \textit{reduced} if either $\vert V(\mathcal C)\vert =1$ 
or no element of $V(\mathcal C)$ has genus 0. Thus, $\mathcal
C$ is: (i) proper and reduced if and only if $V(\mathcal C)$ consists of one
element of genus $0$; (ii) non-proper and reduced if and only if
all elements of $V(\mathcal C)$ are of genus $>0$; (iii) proper and non-reduced
if and only if $V(\mathcal C)$ has more than one element and all of them are of
genus $0$; (iv) non-proper and non-reduced if and only if $V(\mathcal C)$ has
at least one element of genus $0$ and at least one element of
genus $>0$. Note that a proper reduced system of curves on
$\Sigma_g$ contains exactly $g$ curves.

We denote by $G(\mathcal C)$ the graph which is dual to the one
determined by $\mathcal C$ on $\Sigma_g$. Thus, vertices of $G(\mathcal
C)$ correspond to elements of $V(\mathcal C)$ and edges correspond to
curves of $\mathcal C$. Note that loops and multiple edges may arise
in~$G(\mathcal C)$.

A \textit{compression body} $K_g$ of genus $g$ is a 3-manifold with
boundary obtained from $\Sigma_g\times [0,1]$ by attaching a finite
set of 2-handles $Y_1,\ldots, Y_k$ along a system of  curves (called
\textit{attaching circles}) on $\Sigma_g\times\{0\}$ and filling in
with balls all the spherical boundary components of  the resulting
manifold, except $\Sigma_g\times\{1\}$ when $g=0.$
Moreover, $\partial_+ K_g=\Sigma_g\times\{1\}$ is called the
\textit{positive} boundary of $K_g$, while $\partial_- K_g =
\partial K_g-\partial_+ K_g$ is called the \textit{negative}
boundary of $K_g$. Notice that a compression body is a handlebody
if an only if $\partial_- K_g = \emptyset$, i.e., the system of
the attaching circles on $\Sigma_g\times\{0\}$ is proper.
Obviously homeomorphic compression bodies  can be obtained via
(infinitely many) non isotopic systems of attaching circles.

If a system of attaching circles $\mathcal C$ is not reduced, then it
contains at least one reduced subsystem of curves  determining the
same compression body $K_g$. Indeed, let  $V^+(\mathcal C)$ be the set of
vertices of $G(\mathcal C)$ corresponding to the components with genus
greater then zero, and $\mathcal A(\mathcal C)$ be the set
consisting of  all the graphs $T_i$ such that:
\begin{itemize}
\item $T_i$ is a subgraph of $G(\mathcal C)$;
\item if $V^+(\mathcal C)=\emptyset$ then $T_i$ is a maximal tree in
$G(\mathcal C)$;
\item if $V^+(\mathcal C)\ne \emptyset$ then $T_i$ contains all the
vertices of $G(\mathcal C)$ and each component of $T_i$ is a tree
containing exactly one vertex of $V^+(\mathcal C)$.
\end{itemize}
Then, for any  $T_i \in \mathcal A(\mathcal C)$, the system of
curves obtained by removing from $\mathcal C$ the curves
corresponding to the edges of $T_i$ is reduced and determines the
same compression body. Note that this operation corresponds to
removing complementary 2- and 3-handles. Moreover, if  $\partial_-
K_g$ is orientable (resp. non-orientable) and has $h$ boundary
components with genus $g_j$ (resp. $2g_j$), $1\le j \le h$,
then\footnote{The formula corrects a misprint contained in
\cite{[CMV]}.}
$$
\vert E(T_i)\vert = \vert\mathcal C\vert - g - \max\{0,h-1\} +
\sum_{j=1}^h g_j
$$
for each $T_i\in\mathcal A(\mathcal C)$, where $E(T_i)$ denotes the edge
set of
$T_i$.

Let $M$ be  a compact, connected
3-manifold without spherical boundary components. A
\textit{Heegaard surface} for $M$ is a surface $\Sigma_g$
embedded in $M$ such that $M-\Sigma_g$ consists of two components
whose closures $K'$ and $K''$ are (homeomorphic to) a genus $g$
handlebody and  a genus $g$ compression body, respectively.

The triple $(\Sigma_g, K',K'')$ is called a
\textit{Heegaard splitting} of genus $g$ of $M$. It is a well
known fact that each compact connected 3-manifold without
spherical boundary components admits a Heegaard splitting.

\begin{rem}
\textup{By Proposition 2.1.5 of \cite{[Ma]}, the complexity of a
manifold is not affected by puncturing it. So, in order to compute
complexity, there is no loss of generality to assume that the manifold
has no spherical boundary components.}
\end{rem}

On the other hand, a triple $H=(\Sigma_g,
\mathcal C',\mathcal C'')$, where  $\mathcal C'$ and
$\mathcal C''$ are two systems of curves on $\Sigma_g$, such that
they intersect transversally and $\mathcal C'$ is proper, uniquely
determines a 3-manifold $M_H$ corresponding to the Heegaard
splitting $(\Sigma_g, K',K'')$, where $K'$ and $K''$ are respectively
the handlebody and the compression body whose attaching circles
correspond to the curves in the two systems. Such a triple is called a
\textit{generalized Heegaard diagram}  for $M_H$.

In the case of closed 3-manifolds, both systems of curves of a generalized Heegaard diagram $H$ are obviously proper; if they are also reduced, $H$ is simply a Heegaard diagram in the classical sense (see \cite{[He]}).

For each generalized Heegaard diagram  $H=(\Sigma_g,
\mathcal C',\mathcal C'')$, we denote by $\Delta(H)$  the graph embedded in $\Sigma_g$ defined
by the curves of $\mathcal C'\cup \mathcal C''$, and by $\mathcal
R(H)$ the set of regions of $\Sigma_g-\Delta(H)$. Note that $\Delta(H)$ may have connected components which are circles.
All vertices not belonging in these components are 4-valent and they are called \textit{singular}
vertices.  A diagram $H$ is called \textit{reduced Heegaard diagram}
if both the systems of curves are reduced. If $H$ is non-reduced,
then we denote by $\textup{Rd}(H)$ the set of reduced Heegaard
diagrams obtained from $H$ by reducing the two systems of curves.

The modified complexity of a reduced Heegaard diagram  $H'$ is
$$\widetilde{c}(H')=c(H')-\max\,\{n(R)\mid R\in\mathcal R(H')\},$$ where $c(H')$ is the number of singular vertices of $\Delta(H')$ and $n(R)$ denotes the number of singular vertices contained in the
region $R$; while the modified complexity of a (non-reduced) generalized Heegaard
diagram $H$ is $$\widetilde{c}(H)=\min\,\{\tilde{c}(H')\mid
H'\in\textup{Rd}(H)\}.$$

We define the \textit{modified Heegaard complexity} of a compact
connected 3-manifold $M$ as
$$
c_{HM}(M) = \min\,\{\widetilde{c}(H) \mid H
\in\mathcal{H}(M)\} ,
$$
where $\mathcal{H}(M)$ is the set of all generalized Heegaard diagrams of $M$.

\medskip

The significance of  modified Heegaard complexity consists in its
relation with Matveev complexity $c(M)$:
\begin{prop}
If $M$ is a compact connected 3-manifold, then
$$
c(M) \leqslant c_{HM}(M) .
$$
\end{prop}

\begin{proof}
The result has been proved in \cite{[CMV]} for compact orientable
manifolds, but the proof works exactly in the same way also for
non-orientable ones.
\end{proof}

\bigskip
\bigskip

\section{\hskip -0.7cm . Crystallizations and GM-complexity}\label{preliminari cGM}

The present section is devoted to briefly review some basic
notions of the representation theory of PL-manifolds by
crystallizations; in particular, we focus on definitions and
results (due to \cite{[C$_4$]}, \cite{[C$_5$]} and
\cite{[CC$_1$]}) concerning the possibility of obtaining an upper
bound for Matveev complexity of a closed 3-manifold $M$ by means
of the edge-coloured graphs representing $M$.

For general PL-topology and elementary notions about graphs and embeddings, we
refer to \cite{[HW]} and \cite{[Wh]} respectively.

Crystallization theory is a representation tool for general piecewise linear (PL) compact manifolds, without
assumptions about dimension, connectedness, orientability or boundary properties  (see the survey papers \cite{[FGG]}, \cite{[BCG]}
and \cite{[BCCGM]}).
However, since this paper concerns only $3$-manifolds, we will restrict definitions and results to dimension $3$, although they
mostly hold for the general case ($n\ge1$); moreover,
from now on all manifolds will be assumed to be closed and connected.

Given a pseudocomplex $K$, triangulating a $3$-manifold
$M$, a \textit{coloration} on $K$ is a labelling of its vertices
by $\Delta_3=\{0,1,2,3\}$, which is injective on each simplex
of $K$.
The dual 1-skeleton of $K$ is a (multi)graph $\Gamma=(V(\Gamma),E(\Gamma))$ embedded in
$|K|=M$; we can define a map $\gamma :
E(\Gamma)\to\Delta_3$ in the following way: $\gamma(e)=c\ $ iff
the vertices of the face dual to $e$ are coloured by
$\Delta_3-\{c\}.$
The map $\gamma$ - which is injective on each pair of adjacent edges of
the graph - is called an \textit{edge-coloration} on $\Gamma,$
while the pair $(\Gamma,\gamma)$ is called a \textit{4-coloured
graph representing M} or simply a \textit{gem} (where ``gem" stands for  \textit{graph encoded
manifold}: see \cite{[Li]}).
In order to avoid long notations, in the following we will often
omit the edge-coloration when it is not necessary, and we will simply
write $\Gamma$ instead of $(\Gamma,\gamma)$.

It is easy to see that any 3-manifold $M$ has a gem inducing it: just take the
barycentric subdivision $H^{\prime}$ of any pseudocomplex $H$ triangulating $M$,
label any vertex of $H^{\prime}$ with the dimension of the open simplex containing it
in $H$, and construct the associated 4-coloured graph as described above.
Conversely, starting from $\Gamma$, we can always
reconstruct $K(\Gamma)=K$ and hence the manifold $M$ (see
\cite{[FGG]} and \cite{[BCG]} for more details).

Given $i,j\in\Delta_3$, $i\neq j$, we denote by $(\Gamma_{i,j},\gamma_{i,j})$
the $2$-coloured graph obtained from $\Gamma$ by deleting all edges which are not $i$- or
$j$-coloured; hence, $\Gamma_{i,j}=(V(\Gamma),$
$\gamma^{-1}(\{i,j\}))$ and
$\gamma_{i,j}=\gamma_{|_{\gamma^{-1}(\{i,j\})}}.$
The connected components of $\Gamma_{i,j}$ will be
called \textit{$\{i,j\}$-residues} of $\Gamma$ and their number
will be denoted by $g_{i,j}$.
As a consequence of the definition, a bijection is established
between the set of $\{i,j\}$-residues of $\Gamma$
and the set of 1-simplices of $K(\Gamma)$ whose endpoints
are labelled by $\Delta_3-\{i,j\}$.
Moreover, for each $c\in\Delta_3$, the connected components of the
3-coloured graph $\Gamma_{\hat c}$ obtained from $\Gamma$ by
deleting all $c$-coloured edges are in bijective correspondence
with the $c$-coloured vertices of $K(\Gamma)$; their number will be denoted by $g_{\hat c}.$
We will call $\Gamma$ \textit{contracted} iff $\Gamma_{\hat c}$ is connected
for each $c\in\Delta_3$, i.e. iff $K(\Gamma)$ has exactly four
vertices.

A contracted 4-coloured graph representing a $3$-manifold $M$ is called a
\textit{crystallization} of $M$.
It is well-known that every 3-manifold admits a crystallization (see \cite{[FGG]}, together with its references).
Any crystallization (or more generally any gem) $\Gamma$ of $M$ encodes in a combinatorial way the topological properties
of $M$. For example, it is very easy to check that $M$ is orientable iff
$\Gamma$ is bipartite.

\medskip

Relations among crystallization theory and other classical
representation methods for PL manifolds have been widely analyzed
(see \cite[Sections 3, 6, 7]{[BCG]}). In particular, for our
purposes, it is useful to recall how crystallizations and Heegaard
diagrams are strongly correlated.

A cellular embedding $\iota$ of a 4-coloured graph $\Gamma$ into a
surface is said to be {\it regular} if there exists a cyclic
permutation $\varepsilon$ of $\Delta_3$ such that the regions of
$\iota$ are bounded by the images of
$\{\varepsilon_j,\varepsilon_{j+1}\}$-residues of $\Gamma$
($j\in\mathbb Z_4$). If $\Gamma$ is a bipartite (resp.
non-bipartite) crystallization of a 3-manifold $M,$ for each pair
$\alpha, \beta\in \Delta_3,$ let us set
$\{\alpha^{\prime},\beta^{\prime}\}=\Delta_3-\{\alpha ,\beta\}$
and let $F_{\alpha ,\beta}$ be the orientable (resp. non
orientable) surface of genus $g_{\alpha,
\beta}-1=g_{\alpha^{\prime},\beta^{\prime}}-1$, obtained from
$\Gamma$ by attaching a 2-cell to each $\{i,j\}$-residue such that
$\{i,j\}\neq\{\alpha ,\beta\}$ and
$\{i,j\}\neq\{\alpha^{\prime},\beta^{\prime}\}$. This construction
proves the existence of a regular embedding $\iota_{\alpha, \beta}
: \Gamma \to F_{\alpha, \beta}.$ Moreover, if $\mathcal D$ (resp.
$\mathcal D^{\prime}$) is an arbitrarily chosen $\{\alpha,
\beta\}$-residue (resp.
$\{\alpha^{\prime},\beta^{\prime}\}$-residue) of $\Gamma$, the
triple $\mathcal H_{\alpha,\beta,\mathcal D,\mathcal
D^{\prime}}=(F_{\alpha, \beta},{\bf x},{\bf y})$, where {\bf x}
(resp. {\bf y}) is the set of the images of all $\{\alpha,
\beta\}$-residues (resp.
$\{\alpha^{\prime},\beta^{\prime}\}$-residues)
of $\Gamma,$ except $\mathcal D$ (resp. $\mathcal D^{\prime}$), is a Heegaard diagram
of $M$. Conversely, given a Heegaard diagram $\mathcal H=(F,{\bf
x},{\bf y})$ of $M$ and $\alpha, \beta\in \Delta_3,$ there exists
a construction which, starting from $\mathcal H$ yields a
crystallization $\Gamma$ of $M$ such that $\mathcal H=\mathcal
H_{\alpha,\beta,\mathcal D,\mathcal D^{\prime}}$ for a suitable
choice of $\mathcal D$ and $\mathcal D^{\prime}$ in $\Gamma$ (see
\cite{[G$_2$]}).

Now, let us denote by $\mathcal R_{\mathcal D, \mathcal
D^{\prime}}$ the set of regions of $F_{\alpha, \beta}-({\bf x}\cup
{\bf y})=F_{\alpha, \beta}-\iota_{\alpha, \beta}((\Gamma_{\alpha,
\beta} - \mathcal D)\cup (\Gamma_{\alpha^{\prime}, \beta^{\prime}} -
\mathcal D^{\prime})).$

\smallskip
\medskip
\par \noindent
{\bf Definition 1.} Let $M$ be a closed 3-manifold, and let
$(\Gamma, \gamma)$ be a crystallization of $M.$  With the above
notations,  the {\it Gem-Matveev complexity} (or {\it GM-complexity},
for short) of $\Gamma$ is defined as the non-negative integer
$$ \hskip -0.3truecm c_{GM}(\Gamma) = \min \{\# V(\Gamma)-\# (V(\mathcal D)\cup V(\mathcal D^{\prime}) \cup V(\Xi)) \
\mid \ \alpha,\beta\in \Delta_3,  \mathcal D \subset
\Gamma_{\alpha,\beta}, \mathcal D^{\prime} \subset
\Gamma_{\alpha^{\prime}, \beta^{\prime}}, \Xi \in \mathcal
R_{\mathcal D, \mathcal D^{\prime}}\},$$ while the {\it
(non-minimal) GM-complexity} of $M$ is defined as
 the minimum value of {\it GM-complexity}, where the minimum is taken over all\footnote{Note that the original paper \cite{[C$_4$]}
 introduces also the notion of  {\it Gem-Matveev complexity} (or {\it GM-complexity}  for short) of $M$ - denoted by $c_{GM}(M)$ - as the minimum value of GM-complexity, where the minimum is taken only over crystallizations of $M$ which are \underline{minimal} with respect to the order of the graph. Obviously, $c^{\prime}_{GM}(M) \le c_{GM}(M)$
 for every $M$.}
crystallizations of $M:$
 $$ c^{\prime}_{GM}(M) = \min \{c_{GM}(\Gamma) \
\mid \ (\Gamma, \gamma)
 \ \text{crystallization \ of \ } M \}.$$

\bigskip
The following key result, due to \cite{[C$_4$]}, justifies the choice of
terminology:

\begin{prop} \label{relazione c/cGM}
For every closed 3-manifold $M$, Gem-Matveev complexity gives an
upper bound for Matveev complexity of $M$:
$$c(M) \le c^{\prime}_{GM}(M).$$
\end{prop}

\bigskip

Unfortunately, the edge-coloured graphs which are obtained, by
suitable constructions, from different representations of
manifolds are mostly non-contracted. Therefore, the above
definitions need slight modifications in order to be useful for
the general case of a non-contracted gem $(\Gamma, \gamma)$ of
$M.$

\noindent For each pair $\alpha, \beta\in \Delta_3,$ let
$K_{\alpha, \beta}$ be the 1-dimensional subcomplex of $K(\Gamma)$
generated by the  $\{\alpha, \beta\}$-coloured vertices. Moreover,
let $\mathbf D=\{\mathcal D_1, \dots, \mathcal D_{g_{\hat \alpha}
+g_{\hat \beta} -1}\}$ (resp. $\mathbf D^{\prime}=\{\mathcal
D_1^{\prime}, \dots, \mathcal D_{g_{\widehat \alpha^{\prime}}
+g_{\widehat \beta^{\prime}} -1}^{\prime}\}$) be a collection of
$\{\alpha, \beta\}$-coloured (resp. $\{\alpha^{\prime},
\beta^{\prime}\}$-coloured) cycles of $(\Gamma, \gamma)$
corresponding to a maximal tree of $K_{\alpha, \beta}$ (resp.
$K_{\alpha^{\prime}, \beta^{\prime}}$); we denote by $\mathcal
R_{\mathbf D, \mathbf D^{\prime}}$ the set of regions of
$F_{\alpha, \beta}-\iota_{\alpha, \beta}((\Gamma_{\alpha, \beta} -
\cup_{i=1, \dots, g_{\hat \alpha} +g_{\hat \beta} -1} \mathcal
D_i)\cup (\Gamma_{\alpha^{\prime}, \beta^{\prime}} - \cup_{j=1,
\dots, g_{\widehat \alpha^{\prime}} +g_{\widehat \beta^{\prime}}
-1} \mathcal D_j^{\prime})),$ $\iota_{\alpha, \beta} : \Gamma \to
F_{\alpha, \beta}$ being a regular embedding of $\Gamma$ into the
(orientable or non-orientable, according to the bipartition of
$\G$) closed surface of genus $g_{\alpha, \beta}- g_{\hat
\alpha}-g_{\hat \beta} +1.$

\bigskip

\par \noindent
{\bf Definition 2.} Let $M$ be a closed 3-manifold and let
$(\Gamma, \gamma)$ be an edge-coloured graph representing $M.$
With the above notations, the {\it GM-complexity} of $\Gamma$ is
defined as the non-negative integer
$$\begin{aligned}  c_{GM}(\Gamma) =
\min  \{\# V(\Gamma)-\# & \left[\left(\bigcup_{\mathcal D_i \in
\mathbf D} V(\mathcal D_i) \right) \cup \left(\bigcup_{\mathcal
D_j^{\prime} \in \mathbf D^{\prime}} V(\mathcal
D_j^{\prime})\right) \cup V(\Xi)\right] \ \mid \ \alpha,\beta\in
\Delta_3,  \\ \ & \hskip 4
 truecm \mathbf D \subset \Gamma_{\alpha,\beta},  \   \mathbf
D^{\prime} \subset \Gamma_{\alpha^{\prime},\beta^{\prime}}, \ \ \Xi \in
\mathcal R_{\mathbf D, \mathbf D^{\prime}} \}.
\end{aligned}$$

\bigskip

Note that if $\Gamma$ is contracted, the maximal tree of $K_{\alpha, \beta}$ (resp.
$K_{\alpha^{\prime}, \beta^{\prime}}$) consists of one edge, therefore  $\mathbf D$ (resp. $\mathbf D^{\prime}$) contains exactly one $\{\alpha, \beta\}$-coloured (resp. $\{\alpha^{\prime},
\beta^{\prime}\}$-coloured) cycle. Hence the above definition agrees with Definition 1 in the case of crystallizations.

\bigskip

\noindent {\bf Definition 3.} Let $M$ be a closed 3-manifold.
The {\it extended GM-complexity} of $M$ is defined as the minimum
value of {\it GM-complexity}, where the minimum is taken over
\underline{all} edge-coloured
graphs representing $M$ (without assumptions about
contractedness):
$$ \tilde c_{GM}(M) = \min \left\{c_{GM}(\Gamma) \
\mid \ |K(\Gamma)|=M \right\}.$$

\bigskip

The following result - due to \cite{[C$_5$]} - allows to consider
non-minimal $GM$-complexity and extended $GM$-complexity as
``improvements" of Gem-Matveev one, in order to estimate Matveev
complexity.

\begin{prop} \label{relazione c/ctildeGM}
For every closed 3-manifold $M$, the following chain of
inequalities holds:
$$c(M) \le \tilde c_{GM}(M) \le c^{\prime}_{GM}(M).$$
\end{prop}

\bigskip
\bigskip

\section{\hskip -0.7cm . Proof of the main result}\label{paragrafo_uguaglianza}

In this section we prove the announced equality between the two
(apparently) different approaches to Matveev complexity described
in the previous sections.

\begin{lemma} \label{disuguaglianza_I}
For every closed 3-manifold $M$, the following inequality holds:
$$ c^{\prime}_{GM}(M) \le c_{HM}(M).$$
\end{lemma}

\begin{proof}
First of all, we can suppose $c_{HM}(M)\ne 0$, since for each
closed 3-manifold with Matveev complexity zero
$c^{\prime}_{GM}(M)=0$ holds (see \cite{[C$_4$]} and
\cite{[CC$_1$]}).

Let now $\bar H= (\Sigma_g, \mathcal{C'},\mathcal{C''})$ be a
generalized Heegaard diagram of $M$, such that $c_{HM}(\bar H)=
c_{HM}(M) = \bar c$. By definition, there exists a reduced
Heegaard diagram $\bar H^{\prime}$ of $M$ ($\bar
H^{\prime}\in\textup{Rd}(\bar H)$), with $c_{HM}(\bar H^{\prime})=
\bar c$; let $\bar R\in \mathcal{R}(\bar H^{\prime})$ be the
region of $\Delta (\bar H^{\prime})$ such that $c_{HM}(\bar
H^{\prime})=c(\bar H^{\prime})- n(\bar R)$ ($n(\bar R)$ being the
number of singular vertices contained in $\bar R$).

We are going to apply to $\bar H^{\prime}$ the construction
described in \cite[Lemma 4]{[G$_2$]}.
Note that the hypothesis $c_{HM}(\bar H^{\prime})= c_{HM}(M)$ directly
implies that $\bar H^{\prime}$ satisfies condition (a) of \cite[Lemma
4]{[G$_2$]}. Moreover, condition (b) of
the cited Lemma may also be assumed without affecting $c_{HM}(\bar H^{\prime}).$

\smallskip
Let us first suppose that the reduced  Heegaard diagram $\bar
H^{\prime}$  is such that  the graph $\Gamma^{\prime}$ imbedded in
$\Lambda_{2g}$ consisting of all the curves of $\mathcal{C''}$ and
two copies of each curve of $\mathcal{C'}$ (see \cite[p.
476]{[G$_2$]} for details) is connected. In this case, $\bar
H^{\prime}$ trivially satisfies also condition (c) of \cite[Lemma
4]{[G$_2$]}, unless  there is no intersection between the curves
of $\mathcal{C'}$ and $\mathcal{C''}$  (i.e. unless $\bar
H^{\prime}$ contradicts the hypothesis $c_{HM}(M) \ne 0$). As a
consequence, it is possible to construct a crystallization $\bar
\Gamma$  of $M$, such that one of its associated Heegaard diagrams
is exactly $\bar H^{\prime}:$ this means that a  $\{0,
2\}$-residue $\mathcal D$ (resp. a $\{1,3\}$-residue $\mathcal
D^{\prime}$) of $\bar \Gamma$ exists so that $\bar H^{\prime} =
\mathcal H_{0,2,\mathcal D,\mathcal D^{\prime}}=(F_{0, 2},{\bf
x},{\bf y})$, where $F_{0,2}$ is the surface of genus $g_{0,2}-1$
into which $\bar \Gamma$ regularly embeds via $\iota_{0,2}$ and
{\bf x} (resp. {\bf y}) is the set of the images in $\iota_{0,2}$
of all $\{0, 2\}$-residues (resp. $\{1, 3\}$-residues) of $\bar
\Gamma$ but $\mathcal D$ (resp. $\mathcal D^{\prime}$) (see
Section 2). It is now easy to check that $\bar R\in
\mathcal{R}(\bar H^{\prime})$ corresponds to a region $\Xi \in
\mathcal R_{\mathcal D, \mathcal D^{\prime}}$, where $\mathcal
R_{\mathcal D, \mathcal D^{\prime}}$ denotes the set of regions of
$F_{0,2}-({\bf x}\cup {\bf y}).$ Hence, by definition, $
c_{GM}(\bar \Gamma) \le  \# V(\bar \Gamma)-\# (V(\mathcal D)\cup
V(\mathcal D^{\prime}) \cup V(\Xi))= c(\bar H^{\prime})- n(\bar
R)= \bar c.$ In this case, the thesis $c^{\prime}_{GM}(M) \le
c_{HM}(M)$ directly follows.

\smallskip
Let us now assume that the reduced  Heegaard diagram $\bar H^{\prime}$ is
such that the graph $\Gamma^{\prime}$ imbedded in $\Lambda_{2g}$ is not
connected.
If  $\Gamma^{\prime}_1,$ $\Gamma^{\prime}_2,$ $\dots,$ $\Gamma^{\prime}_h$
($h \ge 2$) denote its connected components, then the reduced  Heegaard
diagram $\bar H^{\prime}$ splits into $h$ Heegaard diagrams $\bar
H^{\prime}_1,$  $\bar H^{\prime}_2,$ $\dots,$  $\bar H^{\prime}_h,$ where
$\bar H^{\prime}_i= (\Sigma_{g_i},
\mathcal{C}_i',\mathcal{C}_i'')$ is such that $\sum_{i=1}^h g_i =g,$
$\Sigma_{g}= \#_{i=1}^h  \Sigma_{g_i}$ and $\cup_{i=1}^h \mathcal{C}_i' =
\mathcal{C'}$ (resp. $\cup_{i=1}^h \mathcal{C}_i'' = \mathcal{C''}$).
Note that $\bar R$ is the only region of $\Delta (\bar H^{\prime})$
obtained by ``fusing" the regions $\bar R_1,$ $\bar R_2,$ $\dots,$ $\bar
R_h$   ($\bar R_i$ being a suitable region of $\Delta (\bar H^{\prime}_i)$
with $n(\bar R_i) \ne 0$ singular vertices, and $ \sum_{i=1}^h n(\bar
R_i)= n(\bar R)$).\footnote{Roughly speaking, we can say that $\bar R$ is the ``external"
region of the embedding of the Heegaard diagram $\bar H^{\prime}$  in
$\Lambda_{2g}$, and that $\bar R_i$ is the ``external" region of the
embedding of the Heegaard diagram $\bar H^{\prime}_i$ in $\Lambda_{2g_i}$,
for each $i=1, \dots, h$.}
In fact, if this is not the case, it is easy to check that a new Heegaard
diagram $\bar H^{\prime\prime}$
of $M$ with this property exists, with  $c_{HM}(\bar H^{\prime \prime}) <
c_{HM}(\bar H^{\prime}.)$
Moreover, $c_{HM}(\bar H^{\prime}_i)=c(\bar H^{\prime}_i)- n(\bar R_i)$
trivially holds, together with $c_{HM}(\bar H^{\prime}_i)= c_{HM}(M_i)$,
$M_i$ ($i=1, \dots, h$) being the 3-manifold represented by the Heegaard
diagram $\bar H^{\prime}_i,$ so that $M = \#_{i=1}^h M_i.$  Hence,
$c_{HM}(M)= \sum_{i=1}^h c_{HM}(M_i).$

On the other hand, if $\bar \Gamma$ (resp. $\bar \Gamma^{(i)},$ $\forall
i=1, \dots h$) is the crystallization of $M$ (resp. of $M_i$) obtained
from $\bar H^{\prime}$ (resp. $\bar H^{\prime}_i$) by  the procedure  of
\cite[Lemma 4]{[G$_2$]}, then $\bar \Gamma$ may be trivially obtained by
graph connected sum (see \cite{[FGG]}) from $\bar \Gamma^{(1)},$ $\bar
\Gamma^{(2)},$ $\dots,$ $\bar \Gamma^{(h)}.$

Now, since $\bar H^{\prime}_i$ is such that the graph
$\Gamma^{\prime}_i$ is connected, the above discussion ensures $
c_{GM}(\bar \Gamma^{(i)}) \le c_{HM}(\bar H^{\prime}_i),$ for each
$i=1, \dots, h$.

Finally, $c_{GM}(\bar \Gamma) \le \sum_{i=1}^h c_{GM}(\bar
\Gamma^{(i)})$ trivially holds by construction. The thesis now
directly follows: $c^{\prime}_{GM}(M) \le c_{GM}(\bar \Gamma) \le
\sum_{i=1}^h c_{GM}(\bar \Gamma^{(i)}) = \sum_{i=1}^h
c^{\prime}_{GM}(M_i) = c_{HM}(M).$
\end{proof}

\medskip

\begin{lemma} \label{disuguaglianza_II}
For every closed 3-manifold $M$, the following inequality holds:
$$ c_{HM}(M) \le \tilde c_{GM}(M).$$
\end{lemma}

\begin{proof}
Let $\G$ be a gem of $M$ such that $c_{GM}(\G)=\tilde c_{GM}(M)$
and let $\alpha, \beta\in\Delta_3$ be such that the minimal
$GM$-complexity of $\G$ is obtained by means of the regular
embedding associated to $\alpha$ and $\beta$.
Moreover, if $K_{\alpha,\beta}$ (resp. $K_{\alpha^{\prime},\beta^{\prime}}$) is the 1-dimensional subcomplex of $K(\Gamma)$
generated by the $\{\alpha,\beta\}$-coloured
(resp. $\{\alpha^{\prime},\beta^{\prime}\}$-coloured) vertices, there
exist a maximal tree of $K_{\alpha,\beta}$ (resp.
$K_{\alpha^{\prime},\beta^{\prime}}$) and an element $\Xi\in
\mathcal R_{\mathbf D, \mathbf D^{\prime}}$ ($\mathbf D$ and
$\mathbf D^\prime$ being the collections of $\{\alpha,\beta\}$-
and $\{\alpha^{\prime},\beta^{\prime}\}$-coloured cycles
corresponding to the maximal trees of
$K_{\alpha^{\prime},\beta^{\prime}}$ and $K_{\alpha,\beta}$
respectively) such that
$$c_{GM}(\G)=\# V(\Gamma)-\# (V(\mathbf D)\cup V(\mathbf D^{\prime}) \cup V(\Xi)),$$

\noindent where $V(\mathbf D)$ (resp. $V(\mathbf D^{\prime})$)
denotes the set of the vertices of all the cycles in $\mathbf D$
(resp. $\mathbf D^{\prime}$).

Let $\bar K$ be the largest 2-dimensional subcomplex of the first
barycentric subdivision of $K(\Gamma)$ disjoint from the first
barycentric subdivisions of $K_{\alpha,\beta}$ and
$K_{\alpha^{\prime},\beta^{\prime}}.$ The surface $F$,
triangulated by $\bar K$, splits $K(\Gamma)$ into two polyhedra
$\mathcal A_{\alpha,\beta}$ and $\mathcal
A_{\alpha^{\prime},\beta^{\prime}}$ whose intersection is exactly
$F$. Moreover
$F=F_{\alpha,\beta}=F_{\alpha^{\prime},\beta^{\prime}}$ (where -
according to the previous section - $F_{\alpha,\beta}$ and
$F_{\alpha^{\prime},\beta^{\prime}}$ are the surfaces into which
$\Gamma$ regularly embeds via $\iota_{\alpha,\beta}$ and
$\iota_{\alpha^{\prime},\beta^{\prime}}$ respectively).

Both $\mathcal A_{\alpha,\beta}$ and $\mathcal
A_{\alpha^{\prime},\beta^{\prime}}$ are compression bodies. In
fact, we can think $\mathcal A_{\alpha,\beta}$ (resp. $\mathcal
A_{\alpha^{\prime},\beta^{\prime}}$) as constructed by considering
a collar of $F$ and by adding on $F\times \{1\}$ the 2-handles
whose attaching spheres are all the
$\{\alpha^{\prime},\beta^{\prime}\}$-coloured (resp.
$\{\alpha,\beta\}$-coloured) cycles of $\Gamma$, except those of
$\mathbf D^\prime$ (resp. $\mathbf D$).

Therefore, we consider the generalized Heegaard diagram of $M$ given by
$H=(F,\mathcal C^\prime,\mathcal C^{\prime\prime})$, where $\mathcal C^\prime$ and $\mathcal C^{\prime\prime}$
are the systems of curves on $F$ defined by the attaching cycles described above.

Actually, $\mathcal A_{\alpha,\beta}$ (resp. $\mathcal
A_{\alpha^{\prime},\beta^{\prime}}$) is a handlebody of genus
$g_{\alpha^{\prime},\beta^{\prime}}-\#\mathbf D^\prime$ (resp.
$g_{\alpha,\beta}-\#\mathbf D$), since it collapses to the graph
$K_{\alpha,\beta}$ (resp. $K_{\alpha^{\prime},\beta^{\prime}}$).

As a consequence
$g(F)=g_{\alpha^{\prime},\beta^{\prime}}-\#\mathbf
D^\prime=g_{\alpha,\beta}-\#\mathbf D$ and both  $\mathcal
C^\prime$ and $\mathcal C^{\prime\prime}$ are proper and reduced.
Since the number of singular vertices of $H$ is exactly $\#
V(\Gamma)-\# (V(\mathbf D)\cup V(\mathbf D^{\prime}))$ and $\Xi$
obviously corresponds to a region of $H$ having the maximal vertex
number, we have $c_{HM}(H)=c_{GM}(\G)$, hence $c_{HM}(M)\leq
\tilde c_{GM}(M).$
\end{proof}

\medskip

\begin{rem}
\textup{The proof of Lemma \ref{disuguaglianza_II} shows that {\it
any gem $\Gamma$ of a closed 3-manifold $M$ induces three
generalized Heegaard diagrams for $M$}, one for each choice of a
pair of different colours $\alpha, \beta\in \Delta_3$. Moreover,
the sets of all $\{\alpha,\beta\}$- and all
$\{\alpha^{\prime},\beta^{\prime}\}$-cycles of $\Gamma$ are two
proper systems of curves on the surface $F_{\alpha ,\beta}$, which
are always non-reduced. In the case of a crystallization, a
reduced diagram may be simply obtained by removing an arbitrary
curve from both systems of curves (i.e. the sets $\mathbf D$ and
$\mathbf D^\prime$ are the smallest possible, each consisting of
only one element).}
\end{rem}

As a direct consequence of Lemma \ref{disuguaglianza_I} and Lemma \ref{disuguaglianza_II}, together with Proposition \ref{relazione c/ctildeGM},
the equality among the three notions follows:

\begin{prop} \label{uguaglianza}
For every closed 3-manifold $M$,
$$  c_{HM}(M) = c^{\prime}_{GM}(M) = \tilde c_{GM}(M).$$
\end{prop}

Hence, both  modified Heegaard complexity and non-minimal
$GM$-complexity and  extended $GM$-complexity turn out to be different
tools to compute the same upper bound for Matveev complexity.

Actually, by experimental results of \cite{[C$_2$]} and \cite{[BCrG$_1$]}, this upper bound is proved to be sharp (i.e.: $c(M) = c_{HM}(M) = c^{\prime}_{GM}(M) = \tilde c_{GM}(M)$)   for the thirty-eight  (resp. sixteen)
closed connected prime orientable (resp. non-orientable)
3-manifolds admitting a coloured triangulation with at most $26$ (resp. $30$)
tetrahedra. As far as we know, there is no example where the
strict inequality holds.

Hence, we formulate the following:

\begin{conj}
For every closed connected  3-manifold $M$,
$$c(M) = c_{HM}(M) = c^{\prime}_{GM}(M) = \tilde c_{GM}(M).$$
\end{conj}

\bigskip
\bigskip

{\it Acknowledgement.} Work performed under the auspices of the
G.N.S.A.G.A. of I.N.d.A.M. (Italy) and financially supported by
M.I.U.R. of Italy, University of Modena and Reggio Emilia and
University of Bologna, funds for selected research topics.

\end{document}